\documentclass{amsart}
\usepackage{amsthm, amssymb, amsmath, pb-diagram}

\input xy
\xyoption{all}

\newcommand{\quash}[1]{}  %%Anything in \quash is ignored

\DeclareMathOperator{\SL}{SL}
\DeclareMathOperator{\Hom}{Hom}

\DeclareMathOperator{\Sym}{Sym}
\DeclareMathOperator{\Spec}{Spec}

\DeclareMathOperator{\id}{id}

\DeclareMathOperator{\Pic}{Pic}
\DeclareMathOperator{\ch}{\mathfrak{c}}
\DeclareMathOperator{\Span}{Span}
\DeclareMathOperator{\Rig}{Rig}

\def\AA{\mathbb{A}}
\def\BB{\mathbb{B}}
\def\CC{\mathbb{C}}
\def\EE{\mathbb{E}}

\def\GG{\mathbb{G}}

\def\PP{\mathbb{P}}
\def\QQ{\mathbb{Q}}
\def\RR{\mathbb{R}}
\def\ZZ{\mathbb{Z}}

\def\frakg{\mathfrak{g}}
\def\frakt{\mathfrak{t}}
\def\frakh{\mathfrak{h}}

\def\calO{\mathcal{O}}

\def\calL{\mathcal{L}}

\def\Grass{\mathcal{G}r}
\def\Flag{\mathcal{F}\ell}

\def\tilT{\widetilde{T}}
\def\tilt{\widetilde{\frakt}}
\def\tilw{\widetilde{w}}
\def\tilW{\widetilde{W}}

\swapnumbers
\theoremstyle{plain}
\newtheorem{theorem}[subsection]{Theorem}
\newtheorem{lemma}[subsection]{Lemma}
\newtheorem{cor}[subsection]{Corollary}
\newtheorem{prop}[subsection]{Proposition}

\theoremstyle{definition}

\theoremstyle{remark}
\newtheorem*{remark}{Remark}

\numberwithin{equation}{subsection}

\title{Goresky-MacPherson Calculus\\ for the Affine Flag Varieties}
\author{Zhiwei Yun}
\address{Princeton University}
\email{zyun@math.princeton.edu}
\date{October 2007}

\begin{document}

\begin{abstract}
We use the fixed point arrangement technique developed by Goresky-MacPherson in \cite{GM} to calculate the part of the equivariant cohomology of the affine flag varieties generated by degree 2. This turns out to be a quadric cone. We also describe the spectrum of the full equivariant cohomology ring as an explicit geometric object. We use our results to show that the vertices of the moment map images of the affine flag varieties lie on a paraboloid.
\end{abstract}

\maketitle

\section{Introduction}

\subsection{Goresky-MacPherson calculus} In their paper \cite{GM}, Goresky and MacPherson consider a complex algebraic torus $T$ acting on a complex projective variety $X$. Let $\widehat{H}^*_T(X)$ be the part of the equivariant cohomology $H^*_T(X)$ generated by degree 0 and degree 2. The affine scheme $\Spec \widehat{H}^*_T(X)$ is naturally a closed subscheme of $H_2^T(X)$. If $T$ acts with finitely many fixed points and $X$ is {\em equivariantly formal} (see \cite{GKM}, section 1.2), Goresky and MacPherson show that $\Spec H^*_T(X)$ is reduced and consists of an arrangement of linear subspaces $V_x\subset H^T_2(X)$, one for each fixed point $x$. Here $V_x$ is the image of the natural map induced by the inclusion $i_x:\{x\}\hookrightarrow X$:
\begin{equation}\label{eq:definesection}
\mu_x=i_{x,*}: H^T_2(\{x\})\cong \frakt \to H^T_2(X)
\end{equation}
where $\frakt$ is the Lie algebra of $T$. We call $V=V(T\curvearrowright X)$ the {\em fixed point arrangement} of the action $T$ on $X$.

\subsection{Main results} In this note we extend this result to certain ind-schemes $X$. In examples such as the affine flag varieties, $H^T_2(X)$ is still finite dimensional but there are infinitely many (still discrete) fixed points. The fixed point arrangement $V\subset H^T_2(X)$ still makes sense in this situation. In theorem \ref{th:closure}, we show that the subscheme $\Spec H^*_T(X)$ of $H^T_2(X)$ is the Zariski closure of $V$. We then calculate explicitly the fixed point arrangement of the $\tilT$-action on $\Flag_G$. Here $\tilT$, the {\em augmented torus}, is the product of a maximal torus $T$ of $G$ and the loop-rotation torus $\GG^{rot}_m$. We show in theorem \ref{th:cone} that $\Spec\widehat{H}^*_{\tilT}(\Flag_G)$ is a quadric cone in $H^{\tilT}_2(\Flag_G)$ cut out by the one-dimensional kernel of the cup-product
\begin{equation}\label{eq:cup}
\cup: \Sym^2(H^2_{\tilT}(\Flag_G))\to H^4_{\tilT}(\Flag_G).
\end{equation}

In proposition \ref{p:ultrah2}, we describe $H^{\tilT}_2(\Flag_G)$ in more intrinsic terms using Kac-Moody algebras. We show that, in analogy with the finite flag varieties, the fixed point arrangement are essentially the union of graphs of affine Weyl group action on the affine Cartan algebra.

\subsection{Other results and applications} We will also describe the full equivariant cohomology ring $H^*_{\tilT}(\Flag_G)$. It is convenient to view this equivariant cohomology as the $\GG_m^{rot}$-equivariant cohomology of the stack $[I\backslash G(F)/I]$. The convolution structure realizes $\Spec H^*_{\GG_m^{rot}([I\backslash G(F)/I])}$ as a groupoid over $\AA^1\times\frakt$. We will show in theorem \ref{th:cohoFlag} that there is a natural isomorphism of groupoids over $\AA^1\times\frakt$:
\begin{equation*}
\Spec H^*_{\GG_m^{rot}}([I\backslash G(F)/I])\cong N(\frakt\times_{\ch}\frakt,\frakt\times\frakt) 
\end{equation*}
where $N(\frakt\times_{\ch}\frakt,\frakt\times\frakt)$ is the total space of the deformation to the normal cone of the fiber product $\frakt\times_{\ch}\frakt$ inside the (absolute) product $\frakt\times\frakt$, which carries a natural groupoid structure over $\AA^1\times\frakt$.

One application is to prove the folklore theorem (corollary \ref{c:paraboloid}) saying that the moment map image of the $\tilT$-fixed points of $\Flag_G$ lie on a paraboloid in $\tilt_\RR^*$. A similar statement for the affine Grassmannian was proved by Atiyah and Pressley, see \cite{AP}.

The results here have some overlaps with the results of Kostant-Kumar in \cite{KK}. However, we approach the problem from a different perspective.

\subsection{Notations and conventions} Throughout this note, all (ind-)schemes are over $\CC$. All (co)homologies are taken with $\CC$-coefficients unless otherwise specified (although most results are true for $\QQ$-coefficients).

For an algebraic torus $T$, we denote its character lattice by $X^\bullet(T)$ and cocharacter lattice by $X_\bullet(T)$.

For a graded algebra $H^*$, $\widehat{H}^*$ denotes the subalgebra of $H^*$ generated by degree 0 and 2. 

For an algebraic group $G$, $\BB G$ denotes the classifying stack of $G$ and $\EE G$ denotes the universal $G$-torsor over $\BB G$. 

\subsection{Acknowledgment} This work owes a great deal to Bob MacPherson. The very idea of considering the Zariski closure of the fixed point arrangement in the case of infinitely many fixed points is due to him. The author is grateful to Mark Goresky for extremely helpful suggestions on the presentation of the note. The author also thanks Roman Bezrukavnikov, David Treumann and Xinwen Zhu for stimulating discussions.

\section{Equivariant cohomology of ind-schemes}\label{s:closure}

\subsection{Assumptions on the space and the action}\label{ss:assumptions}

Let $X=\bigcup X_n$ be a strict ind-projective scheme, i.e., each $\iota_n: X_n\to X_{n+1}$ is a closed embedding of projective schemes. Let $T$ be a torus acting algebraically on $X$. We make the following assumptions:
\begin{enumerate}
 \item Assume each $X_n$ is stable under $T$ and $T$-equivariantly formal;
 \item Assume each $X_n$ has only finitely many fixed points;
 \item Assume that $H_2(X)$ is finite dimensional. 
\end{enumerate}

Following the idea of R.MacPherson, we extend (part of) the main theorem of \cite{GM} to the ind-scheme case.

%% GM calculus for ind-schemes

\begin{theorem}\label{th:closure}
Suppose the $T$-action on a connected ind-projective scheme $X$ satisfies the assumptions in section \ref{ss:assumptions}. Then $\Spec \widehat{H}^*_T(X)$ is the Zariski closure of the fixed point arrangement $V$ inside the affine space $H^T_2(X)$ with the reduced scheme structure. 
\end{theorem}
\begin{proof}
Let $i_n:X_n\to X$ be the inclusion map. For each $n$ consider the natural commutative diagram
\[
\xymatrix{H^*_T(X) \ar[r]^(.35){\rho*}\ar[d]^{i_n^*} & \bigoplus_{x\in X^T}H^*_T(\{x\})\ar[d]\\
          H^*_T(X_n)\ar[r]^(.35){\rho_n^*} & \bigoplus_{x\in X_n^T}H^*_T(\{x\})}
\]
Here the arrows are the obvious restriction maps. By assumption (1) of section \ref{ss:assumptions},  $X_n$ is equivariantly formal so we can apply the localization theorem (see \cite{GKM},Theorem 1.2.2) to conclude that $\rho_n^*$ is injective. This implies that $\rho^*$ is also injective. In fact, for any class $c\in H^*_T(X)$, if $\rho^*(c)=0$, then for all$n$ large enough, $\rho_n^*(i_n^*c)=0$, hence $i_n^*(c)=0$. This implies $c=0$. In particular, $H^*_T(X)$, hence $\widehat{H}^*_T(X)$, are reduced.

Now we have a surjection followed by an injection:
\begin{equation}\label{eq:rings}
\Sym(H^2_T(X))\twoheadrightarrow \widehat{H}^*_T(X) \stackrel{\rho^*}{\hookrightarrow} \oplus_{x\in X^T}H^*_T(\{x\}).
\end{equation}
Taking the spectra of each term of (\ref{eq:rings}), we get
\begin{equation*}
\coprod_{x\in X^T}\frakt\stackrel{\coprod\mu_x}{\longrightarrow}\Spec\widehat{H}^*_T(X)\hookrightarrow H^T_2(X).
\end{equation*}
where the first arrow is dominant and the second arrow is a closed embedding. This shows that the support of $\Spec\widehat{H}^*_T(X)$ in $H^T_2(X)$ is the Zariski closure of $V$. 
\end{proof}

%% This section contains the main theorem and its proof 

\section{The fixed point arrangement of affine flag varieties}\label{s:flagarr}

\subsection{Notations concerning the group $G$}\label{ss:defineflag} Let $G$ be a simple and simply-connected group over $\CC$. Let $r$ be the rank of G. Let $F=\CC((z))$ and $\calO=\CC[[z]]$. Choose a maximal torus $T$ and a Borel subgroup $B$ containing $T$. Let $I\subset G(\calO)$ be the Iwahori subgroup corresponding to $B$. The \textit{affine Grassmannian} is $\Grass_G=G(F)/G(\calO)$. The \textit{affine flag variety} is $\Flag_G=G(F)/I$. These are ind-schemes over $\CC$ and we have a natural projection
\begin{equation}\label{eq:fltogr}
p_G: \Flag_G\to\Grass_G
\end{equation}
whose fiber over the base point of $\Grass_G$ is identified with the finite flag variety $f\ell_G=G/B$. Basic facts about affine flag varieties and affine Grassmannian can be found in \cite{GKMFL}, section 5 and 14.

Let $N_G(T)$ be the normalizer of $T$ in $G$ and $W=N_G(T)/T$ be the Weyl group. Let $\Phi$ be the root system determined by $(G,T,B)$. Let $\theta\in\Phi$ be the highest root. Fix a Killing form $(-|-)$ on $\frakt^*$ such that $(\theta|\theta)=2$. It induces an isomorphism
\begin{equation}\label{eq:definesigma}
\sigma: \frakt\stackrel{\sim}{\to}\frakt^*
\end{equation}
such that $(\xi^*|\sigma(\xi))=\langle \xi^*,\xi\rangle$ for all $\xi\in \frakt$ and $\xi^*\in \frakt^*$. In particular, let $\check{\theta}$ denote the coroot corresponding to $\theta$, we have $\sigma(\check\theta)=\theta$. Via $\sigma$, we also have a Killing form $(-|-)$ on $\frakt$. We use $|-|^2$ to denote the quadratic forms on $\frakt^*$ or $\frakt$ associated to the Killing forms.

\subsection{Notations concerning the affine Kac-Moody algebra}\label{ss:kacmoody} Consider the affine Kac-Moody algebra $L(\frakg)$ associated to $\frakg$. We will follow the notations of \cite{Kac} (chapter 6) which we briefly recall here. The affine Cartan algebra $\frakh$ and its dual $\frakh^*$ have decompositions
\begin{equation}\label{eq:decomp}
\frakh=\CC d\oplus\frakt\oplus\CC K,\frakh^*=\CC\delta\oplus\frakt^*\oplus\CC\Lambda_0
\end{equation} 
where $K$ is the {\em canonical central element}, $d$ the {\em scaling element} and $\delta$ the positive generator of the imaginary roots. The pairing between them extends the natural pairing between $\frakt$ and $\frakt^*$ and satisfies
$$\langle \CC d\oplus\CC K,\frakt^*\rangle=0,\langle\frakt,\CC\delta\oplus\CC\Lambda_0\rangle=0,$$
$$\langle d,\Lambda_0\rangle=\langle K,\delta\rangle=0,\langle d,\delta\rangle=\langle K,\Lambda_0\rangle=1.$$

We extend the Killing form $(-|-)$ to $\frakh$ and $\frakh^*$ by:
$$(K|\frakt\oplus\CC K)=0,(d|\frakt\oplus\CC d)=0,(K|d)=1$$
$$(\delta|\frakt^*\oplus\CC\delta)=0,(\Lambda_0|\frakt^*\oplus\CC\Lambda_0)=0,(\delta|\Lambda_0)=1.$$

Let $\widetilde{\Phi}$ be the affine root system of $L(\frakg)$ with simple roots $\{\alpha_0=\delta-\theta,\alpha_1,\cdots,\alpha_r\}$ and simple coroots $\{\check{\alpha}_0=K-\check{\theta},\check\alpha_1,\cdots,\check\alpha_r\}$. Let $\tilW =X_\bullet(T)\rtimes W$ be the affine Weyl group associated to $L(\frakg)$ with simple reflections $\{s_0,\cdots,s_r\}$ corresponding to the simple roots of $\widetilde{\Phi}$. Let $\ell$ be the length function on $\tilW $ with respect to these simple reflections. 

The group $\tilW $ acts on $\frakh$ via the formula
\begin{equation}\label{eq:weylaction}
\tilw(u,\xi,v)=\left(u,w\xi+u\lambda,v-(w\xi|\lambda)-\frac{u}{2}|\lambda|^2\right).
\end{equation}
where $(u,\xi,v)$ are the coordinate in terms of the decomposition (\ref{eq:decomp}) and $\tilw=(\lambda,w)\in\tilW$.

\subsection{The torus action on $\Flag_G$}\label{ss:torus} The torus $T$ acts on $\Flag_G$ by left translations. The one-dimensional torus $\GG_m$ acts on $\calO$ by dilation $s:z\mapsto sz,s\in\GG_m$. This induces an action of $\GG_m$ on $\Flag_G$, which is the so-called {\em loop rotation}. The rotation torus is denoted $\GG_m^{rot}$. Let $\tilT=\GG_m^{rot}\times T$ be the {\em augmented torus}. We are mainly interested in the $\tilT$-action on $\Flag_G$. 

Let $\tilt$ be the Lie algebra of $\tilT$. We have an identification
\begin{equation}\label{eq:tilt}
\tilt\cong\CC d\oplus\frakt=\frakh/\CC K
\end{equation}
such that $d$ corresponds to the positive generator of $X_\bullet(\GG_m^{rot})$. This induces an identification of dual spaces
\begin{equation}\label{eq:tiltdual}
\tilt^*\cong\CC\delta\oplus\frakt^*\subset\frakh^*
\end{equation}
such that $\delta$ corresponds to the dual of $d$.

\subsection{The $\tilW$-symmetry and fixed points} For each $\lambda\in X_\bullet(T)$, let $z^{-\lambda}$ be the image of $z$ under $-\lambda: \GG_m(F)\to T(F)$. The assignment 
\begin{equation}\label{eq:minusaction}
\lambda\mapsto z^{-\lambda}
\end{equation}
identifies $X_\bullet(T)$ as a subgroup of $T(F)$. We put a minus sign here in order to make some formulas appear nicer in the sequel. Let $\widetilde{N}$ be the product $X_\bullet(T)\rtimes N_G(T)\subset G(F)$. Clearly, $\widetilde{N}$ acts on $\Flag_G$ by left translation. We have a natural exact sequence
$$1\to T\to\widetilde{N}\to\tilW \to 1.$$

\begin{lemma}\label{lem:twistequiv}
The induced action of $\widetilde{N}$ on $H_*^{\tilT}(\Flag_G)$ factors through $\tilW $. The natural exact sequence
\begin{equation}\label{eq:homotwoexact}
 0\to H_2(\Flag_G)\to H_2^{\tilT}(\Flag_G)\stackrel{\pi_1}{\longrightarrow} \tilt\to 0
\end{equation}
is $\tilW $-equivariant. Here the action of $\tilW $ on $H_2(\Flag_G)$ is trivial, and the action on $\tilt$ is induced from the action described by (\ref{eq:weylaction}).
\end{lemma}

\begin{proof}
The action of $\widetilde{N}$ on $\Flag_G$ come from the action of the {\em connected} group ind-scheme $G(F)$, hence by the homotopy invariance of homology, the action of $\widetilde{N}$ on the ordinary homology of $H_*(\Flag_G)$ is trivial. This verifies the equivariance of the first arrow in (\ref{eq:homotwoexact}).

For $(s,t)\in \GG_m^{rot}\times T=\tilT, (\lambda, n)\in X_\bullet(T)\rtimes N_G(T)=\widetilde{N}$ and $x\in \Flag_G$, we have
\begin{eqnarray*}
(\lambda,n)\cdot(s,t)\cdot x &=& z^{-\lambda}nts(x)\\
&=& ntn^{-1}z^{-\lambda}ns(x)\\
&=& ntn^{-1}s^{\lambda}s(z^{-\lambda}nx)\\
&=& (ntn^{-1},s^{\lambda})\cdot(\lambda,n)\cdot x.
\end{eqnarray*}
In particular the action of the neutral component $T$ of $\widetilde{N}$ commutes with the action of $\tilT$. By homotopy invariance, it acts trivially on $H_*^{\tilT}(\Flag_G)$. Therefore the action of $\widetilde{N}$ factors through $\tilW$.

The above calculation, passed on to homology, verifies that the projection $\pi_1$ in (\ref{eq:homotwoexact}) is equivariant.
\end{proof}
\begin{remark}
It is our choice of the minus sign in (\ref{eq:minusaction}) that makes $\pi_1$ equivariant. 
\end{remark}

\subsection{The fixed points} Let $\Flag_G^T$ be the $T$-fixed point subset of $\Flag_G$. Note that it coincides with the $\tilT$-fixed points. The group $\widetilde{N}$ acts on $\Flag_G^T$. This action factors through $\tilW $ such that $\Flag_G^T$ becomes a $\tilW $-torsor. We denote the base point of $\Flag_G$ corresponding to $I$ by $x_0$. By our convention (\ref{eq:minusaction}), if $\tilw=(\lambda,w)$, then $\tilw\cdot x_0=z^{-\lambda}n_w x_0$ for any lift $n_w\in N_G(T)$ of $w$.

\begin{cor}\label{c:twistmu}
For $\tilw\in \tilW $, we have
\begin{equation}\label{eq:actiononmu}
\mu_{\tilw\cdot x_0}=\tilw\circ\mu_{x_0}\circ\tilw^{-1}
\end{equation}
and
\begin{equation}\label{eq:actiononv}
\tilw\cdot V_{x_0}=V_{\tilw\cdot x_0}.
\end{equation}
\end{cor}
\begin{proof}
By lemma \ref{lem:twistequiv}, we have a commutative diagram
\[
\xymatrix{\tilt=H^{\tilT}_2(\{x_0\})\ar[rr]^{\mu_{x_0}}\ar[d]^{\tilw} & & H^{\tilT}_2(\Flag_G)\ar[rr]^{\pi_1}\ar[d]^{\tilw} & & \tilt\ar[d]^{\tilw}\\
\tilt=H^{\tilT}_2(\{\tilw\cdot x_0\})\ar[rr]^{\mu_{\tilw\cdot x_0}} & & H^{\tilT}_2(\Flag_G)\ar[rr]^{\pi_1} & & \tilt}
\]
Here the left vertical map $\tilw_*$ is induced by the restriction of the map $\tilw:\Flag_G\to\Flag_G$ to the base point $x_0$. The compositions of the arrows in both top and bottom rows are identities. This proves (\ref{eq:actiononmu}). The other formula (\ref{eq:actiononv}) follows immediately. 
\end{proof}

\subsection{Description of $H_2(\Flag_G)$}\label{ss:homotwo} The affine flag variety $\Flag_G$ has a stratification by $I$-orbits. Each $I$-orbit contains a unique fixed point $\tilw\cdot x_0$ and has dimension $\ell(\tilw)$ (\cite{GKMFL}, section 14). Therefore the one-dimensional orbits are in one-one correspondence with simple reflections in $\tilW $. Let $C_i$ be the closure of $Is_i\cdot x_0$. For all large enough $I$-orbit closures $\overline{I\tilw\cdot x_0}$, $\{[C_0],\cdots,[C_r]\}$ form a basis of $H_2(\overline{I\tilw\cdot x_0},\ZZ)$. Therefore we have a unique isomorphism
\begin{equation}\label{eq:homotwo}
\phi_*: X_\bullet(T)\oplus\ZZ K=\Span_{\ZZ}\{\check{\alpha_0},\cdots,\check{\alpha_r}\}\stackrel{\sim}{\to}H_2(\Flag_G,\ZZ)
\end{equation}
sending $\check{\alpha_i}$ to the cycle class $[C_i]$.

To state the next theorem, we need to define a $\tilW$-action on $\CC d\oplus\frakt\oplus\frakt\oplus\CC K$ by
\begin{equation}\label{eq:4termaction}
\tilw\cdot(u,\xi,\eta,v)=\left(u,w\xi+u\lambda,\eta,v+(w\xi|\lambda)+\frac{u}{2}|\lambda|^2\right)
\end{equation}
for $(u,\xi,\eta,v)\in \CC d\oplus\frakt\oplus\frakt\oplus\CC K$ and $\tilw=(\lambda,w)\in\tilW$.

\begin{theorem}\label{th:flagarr}
\begin{enumerate}
 \item[]
 \item There is a unique $\tilW$-equivariant isomorphism
\begin{equation}\label{eq:definepsi}
\psi_*: \CC d\oplus\frakt\oplus\frakt\oplus\CC K\stackrel{\sim}{\longrightarrow} H^{\tilT}_2(\Flag_G)
\end{equation}
such that:
\begin{equation}\label{eq:mux0}
\mu_{x_0}(u,\xi)=\psi_*(u,\xi,\xi,0)
\end{equation}
and the following diagram
\begin{equation}\label{eq:diagforpsi}
 \xymatrix{0\ar[r] & \frakt\oplus\CC K \ar[d]^{\phi_{*,\CC}}\ar[r]^(.4){i_{34}} & \CC d\oplus\frakt\oplus\frakt\oplus\CC K\ar[rr]^{p_{12}}\ar[d]^{\psi_*} & & \CC d\oplus\frakt\ar[r]\ar@{=}[d] & 0\\
 0\ar[r] & H_2(\Flag_G)\ar[r] & H^{\tilT}_2(\Flag_G)\ar[rr]^{\pi_1} & & \tilt\ar[r] & 0}
\end{equation} 
is an isomorphism of exact sequences. Here $i_{34}$ means the inclusion of the last two factors and $p_{12}$ means the projection to the first two factor. The bottom sequence is (\ref{eq:homotwoexact}).
 \item For $\tilw=(\lambda,w)\in\tilW$, we have 
\begin{equation}\label{eq:formulaformu}
\mu_{\tilw\cdot x_0}(u,\xi)=\psi_*\left(u,\xi,w^{-1}(\xi-u\lambda),(\xi|\lambda)-\frac{u}{2}|\lambda|^2\right).
\end{equation}
\end{enumerate}
\end{theorem}

Because the proof of the theorem involves some tedious calculation, we postpone it to section \ref{s:proof}.

With explicit knowledge about the fixed point arrangement, we are ready to apply theorem \ref{th:closure}.

\begin{lemma}\label{lem:verifyass} 
The $\tilT$-action on $\Flag_G$ satisfies the assumptions in section \ref{ss:assumptions}.
\end{lemma}
\begin{proof}
The closure of each $I$-orbit of $\Flag_G$ is a projective variety stable under $\tilT$-action. Since each $I$-orbit is an affine space, its cohomology vanishes at odd degrees. Therefore each orbit closure is equivariantly formal. This verifies the assumption (1). 

Since each $I$-orbit contains a unique $T$-fixed point, and each $I$-orbit closure consists of finitely many $I$-orbits, therefore each $I$-orbit closure contains only finitely many $T$-fixed points. This verifies assumption (2).

The verification of (3) is done in section \ref{ss:homotwo}.
\end{proof}

\begin{theorem}\label{th:cone}
Under the isomorphism $\psi_*$, $$\Spec\widehat{H}^*_{\tilT}(\Flag_G)\subset H^{\tilT}_2(\Flag_G)\stackrel{\psi_*^{-1}}{\longrightarrow} \AA^1_d\times\frakt\times\frakt\times\AA^1_K$$ is the nondegenerate quadric cone $Q$ given by the equation
\begin{equation}\label{eq:quadric}
q(u,\xi,\eta,v)=|\xi|^2-|\eta|^2-2uv=0.
\end{equation}
Here $(u,\xi,\eta,v)$ are the coordinates for $\AA^1_d\times\frakt\times\frakt\times\AA^1_K$. 
\end{theorem}
\begin{proof}
Using the explicit formula (\ref{eq:formulaformu}), one easily checks that the image $V_{\tilw\cdot x_0}$ of any $\mu_{\tilw\cdot x_0}$ lies on the quadric cone $Q$. According to theorem \ref{th:closure}, we have to check that the fixed point arrangement $V=\bigcup_{\tilw\in\tilW}V_{\tilw\cdot x_0}$ is dense in $Q$. The projection to the first three factors 
\begin{equation}\label{eq:piQ}
 \pi_{Q}=(\pi_0,\pi_1,\pi_2): Q\to \GG_m\times\frakt\times\frakt
\end{equation}
is birational. Therefore it suffices to check that $\pi(V)$ is dense in $\GG_m\times\frakt\times\frakt$.
 
On the other hand, by theorem (\ref{eq:formulaformu}), $\pi(V)$ contains points of the form $(u,\xi,w^{-1}(\xi-u\lambda))$ for all $\xi\in\frakt,\lambda\in X_\bullet(T)$ and $u\neq0$. For fixed $\xi,u$, taking $w=1$, the set $\{\xi-u\lambda|\lambda\in X_\bullet(T)\}$, being a shifted lattice in $\frakt$, is clearly Zariski dense in $\frakt$. Therefore $\pi(V)$ is Zariski dense in $\GG_m\times\frakt\times\frakt$. This proves the theorem.
\end{proof}

%% The full equivariant cohomology ring

\section{The equivariant cohomology of affine flag varieties}\label{s:fullequiv}

\subsection{Further notations}
Recall that 
\begin{equation}\label{eq:definefi}
H^*(\BB G)\cong \Sym(\frakt^*)^W\cong \CC[f_1,\cdots,f_r]
\end{equation}
where $\deg(f_i)=2d_i+2$ and $d_1,\cdots,d_r$ are the exponents of $\frakg$. We normalize $f_1$ to be the Killing quadratic form $|-|^2$ we chose before. We have a natural identification $\Spec H^*(\BB T)=\frakt$ and $\Spec H^*(\BB G)=\ch$, where $\ch$ is by definition the invariant quotient $\frakt// W$.

The group $G(\calO)$ naturally acts on the affine Grassmannian $\Grass_G$ by left translation and $\GG_m^{rot}$ acts on $\Grass_G$ by loop rotation. The equivariant cohomology $H^*_{G(\calO)\rtimes\GG_m^{rot}}(\Grass_G)$ can be viewed as the $\GG_m^{rot}$-equivariant cohomology of the stack $[G(\calO)\backslash G(F)/G(\calO)]$. We will use the following theorem of Bezrukavnikov-Finkelberg. 

\begin{theorem}[\cite{BeFin}, theorem 1]\label{th:cohoGr}
\begin{enumerate}
 \item[]
 \item We have a natural isomorphism
\begin{equation}\label{eq:cohoGr}
\Spec H^*_{\GG_m^{rot}}([G(\calO)\backslash G(F)/G(\calO)])\cong N(\Delta(\ch),\ch\times\ch)
\end{equation}
where $N(\Delta(\ch),\ch\times\ch)$ is the total space of the deformation to the normal cone of the diagonal copy of $\ch$ inside $\ch\times\ch$.
 \item Under the isomorphism (\ref{eq:cohoGr}), the natural projection $N(\Delta(\ch),\ch\times\ch)\to \AA^1\times\ch\times\ch$ corresponds to the natural morphism
\begin{equation}\label{eq:projgr}
\pi_{\Grass}=(\pi_0,\pi_1,\pi_2):\Spec H^*_{\GG_m^{rot}}([G(\calO)\backslash G(F)/G(\calO)])\to \AA^1\times\ch\times\ch
\end{equation}
Here $\pi_0$ is given by the $\GG_m^{rot}$-action, $\pi_1$ and $\pi_2$ are the left and right $G(\calO)$-quotient structures. 
\end{enumerate}
\end{theorem}

Now back to affine flag varieties. Note that the $T$-equivariant cohomology and the $I$-equivariant cohomology of $\Flag_G$ are the same because $\ker(I\to T)$ is a pro-unipotent algebraic group. Using stack notations, we have
\begin{equation*}
H^*_{\tilT}(\Flag_G)\cong H^*_{\GG_m^{rot}}([I\backslash G(F)/I]).
\end{equation*}

Parallel to the affine Grassmannian case, we have
\begin{theorem}\label{th:cohoFlag}
\begin{enumerate}
 \item[]
 \item We have a natural isomorphism
\begin{equation}\label{eq:cohoFlag}
\Spec H^*_{\GG_m^{rot}}([I\backslash G(F)/I])\cong N(\frakt\times_{\ch}\frakt,\frakt\times\frakt) 
\end{equation}
where $N(\frakt\times_{\ch}\frakt,\frakt\times\frakt)$ is the total space of the deformation to the normal cone of the fiber product $\frakt\times_{\ch}\frakt$ inside $\frakt\times\frakt$.
 \item Under the isomorphism (\ref{eq:cohoFlag}), the natural projection $N(\frakt\times_{\ch}\frakt,\frakt\times\frakt)\to \AA^1\times\frakt\times\frakt$ corresponds to the natural morphism
\begin{equation}\label{eq:projflag}
\pi_{\Flag}=(\pi_0,\pi_1,\pi_2):\Spec H^*_{\GG_m^{rot}}([I\backslash G(F)/I])\to \AA^1\times\frakt\times\frakt.
\end{equation}
Here $\pi_0$ is given by the $\GG_m^{rot}$-action, $\pi_1$ and $\pi_2$ are given by the left and right $I$-quotient structures.
\end{enumerate}
\end{theorem}
\begin{proof}
We have a natural ring homomorphism
\begin{equation*}
\epsilon: H^*(\BB T)\otimes_{H^*(\BB G)}H^*_{\GG_m^{rot}}([G(\calO)\backslash G(F)/G(\calO)])\otimes_{H^*(\BB G)}H^*(\BB T)\to H^*_{\GG_m^{rot}}([I\backslash G(F)/I]) 
\end{equation*}
It suffices to show that $\epsilon$ is an isomorphism. In fact, if it is an isomorphism, after taking spectra, we have
\begin{equation*}
\Spec H^*_{\GG_m^{rot}}([I\backslash G(F)/I])\cong \frakt\times_{\ch}\Spec H^*_{\GG_m^{rot}}([G(\calO)\backslash G(F)/G(\calO)])\times_{\ch}\frakt.
\end{equation*}
By theorem \ref{th:cohoGr}, the term on the right is
\begin{equation*}
\frakt\times_{\ch}N(\Delta(\ch),\ch\times\ch)\times_{\ch}\frakt\cong N(\frakt\times_{\ch}\frakt,\frakt\times\frakt).
\end{equation*}
Assertion (2) is clear from construction.

In order to show that $\epsilon$ is an isomorphism, we decompose it into the composite of several ring homomorphisms:
\begin{eqnarray*}
 & & H^*(\BB T)\otimes_{H^*(\BB G)}H^*_{\GG_m^{rot}}([G(\calO)\backslash G(F)/G(\calO)])\otimes_{H^*(\BB G)}H^*(\BB T)\\
 &\cong & H^*(\BB T)\otimes_{H^*(\BB G)}H^*_{G\times\GG_m^{rot}}(\Grass_G)\otimes_{H^*(\BB G)}H^*(\BB T) \\
 &\stackrel{\epsilon_1\otimes \id}{\to} & H^*_{\tilT}(\Grass_G)\otimes_{H^*(\BB G)}H^*(\BB T)\\
 & = & H^*_{\GG_m^{rot}}([I\backslash G(F)/G(\calO])\otimes_{H^*(\BB G)}H^*(\BB T)\\
 & = & H^*_{\GG_m^{rot}}([I\backslash G(F)])_{G(\calO)}\otimes_{H^*(\BB G)}H^*(\BB T)\\
 &\stackrel{\epsilon_2}{\to} & H^*_{\GG_m^{rot}}([I\backslash G(F)])_T\\
 & = & H^*_{\GG_m^{rot}}([I\backslash G(F)/I]).
\end{eqnarray*}
where
\begin{equation}\label{eq:bc}
\epsilon_1: H^*(\BB T)\otimes_{H^*(\BB G)}H^*_{G\times\GG_m^{rot}}(\Grass_G)\to H^*_{\tilT}(\Grass_G)
\end{equation}
and
\begin{equation}
\epsilon_2: H^*_{\GG_m^{rot}}([I\backslash G(F)])_{G(\calO)}\otimes_{H^*(\BB G)}H^*(\BB T)\to H^*_{\GG_m^{rot}}([I\backslash G(F)])_T
\end{equation}
are natural ring homomorphisms. Sometimes we write a group as the right subscript of the cohomology of some space to emphasize that it acts on the space from the right.

We only need to show that $\epsilon_1$ and $\epsilon_2$ are isomorphisms. Notice that the affine Grassmannian $\Grass_G$ is equivariantly formal under either $\tilT$-action or $G\times\GG_m^{rot}$-action. This can be seen by a similar argument using the $I$-orbits as in the proof of lemma \ref{lem:verifyass}. Now, each side of (\ref{eq:bc}) is equipped with a filtration by the Leray-Serre spectral sequences and $\epsilon_1$ respects these filtrations. Moreover, on the level of the associated graded rings, both sides are canonically identified with $H^*(\BB\tilT)\otimes H^*(\Grass_G)$ by equivariant formality of $\Grass_G$. In other words, $gr(\epsilon_1)$ is an isomorphism. Therefore, $\epsilon_1$ is also an isomorphism. 

The argument for $\epsilon_2$ is the same. This completes the proof.
\end{proof}

\begin{remark}
The component $\pi_1$ of $\pi_{\Flag}$ coincides with the map with the same name defined in lemma \ref{lem:twistequiv}. We will see in the next section (proposition \ref{p:totaltocone}) that $\pi_{\Flag}$ coincides with the map $\pi_{Q}$ defined in (\ref{eq:piQ}).
\end{remark}

By the definition of the deformation to the normal cone (see, for example, \cite{BeFin}, section 2.3), we have the following explicit description:

\begin{prop}\label{p:flagexplicit}
The equivariant cohomology ring $H^*_{\tilT}(\Flag_G)$ is a graded $\CC$-algebra generated by $u,\xi=(\xi_1,\cdots,\xi_r),\eta=(\eta_1,\cdots,\eta_r)$ and $v=(v_1,\cdots,v_r)$ subject to the relations
$$f_i(\xi)-f_i(\eta)=uv_i, \forall i=1,\cdots,r.$$
Here $\deg(\xi_i)=\deg(\eta_i)=\deg(u)=2,\deg(v_i)=2d_i$. The $f_i$'s are fundamental $W$-invariants in $\Sym(\frakt^*)$ as in (\ref{eq:definefi}). In terms of these coordinates, the natural morphism $\pi_{\Flag}$ in (\ref{eq:projflag}) has the form
$$\pi_{\Flag}(u,\xi,\eta,v)=(u,\xi,\eta).$$
\end{prop}

%\subsection{Specializations} 

Specializing the above results to the linear subspace $u=0$, we get the analogous results for the $T$-equivariant cohomology of $\Flag_G$.

\begin{cor}\label{c:Teqcohoflag}
\begin{enumerate}
 \item[]
 \item We have a natural isomorphism
\begin{equation*} 
 \Spec H^*_T(\Flag_G)\cong NC_{\frakt\times_{\ch}\frakt/\frakt\times\frakt}
\end{equation*}
where $NC_{\frakt\times_{\ch}\frakt/\frakt\times\frakt}$ is the normal cone of $\frakt\times_{\ch}\frakt$ in $\frakt\times\frakt$.
 \item In terms of coordinates introduced in proposition \ref{p:flagexplicit}, we have
\begin{equation*}
 H^*_T(\Flag_G)=\CC[\xi,\eta,v]/(f_i(\xi)-f_i(\eta),i=1,\cdots,r).
\end{equation*}
\end{enumerate}
\end{cor}

Further specializing to $\xi=0$, we get the ordinary cohomology of $\Flag_G$.

\begin{cor}\label{c:ordcohoflag}
$$H^*(\Flag_G)=\CC[\eta,v]/(f_i(\eta),i=1,\cdots,r)\cong H^*(f\ell_G)\otimes H^*(\Grass_G).$$
\end{cor}
\begin{remark}
The above isomorphism can also be deduced from the degeneracy of the Leray spectral sequence associated to the fibration $p_G: \Flag_G\to\Grass_G$ at $E_2$-term.
\end{remark}

% comparison and more symmetries

\section{Comparison of two descriptions and more symmetries on $H^*_{\tilT}(\Flag_G)$}\label{s:compsym}

In this section, we first compare the two descriptions of the equivariant cohomology ring of $\Flag_G$ given by theorem \ref{th:cone} and theorem \ref{th:cohoFlag}. We then explore more symmetries of $\Flag_G$ to restate theorem \ref{th:flagarr} in a more intrinsic way. In section \ref{ss:linearinquadric}, we will study the fixed point arrangement from the perspective of classical geometry of quadrics .

\begin{prop}\label{p:totaltocone}
Under the isomorphisms in theorem \ref{th:cone} and theorem \ref{th:cohoFlag}, the natural projection
\begin{equation*}
\Pi: N(\frakt\times_{\ch}\frakt,\frakt\times\frakt)\cong\Spec H^*_{\tilT}(\Flag_G)\to \Spec \widehat{H}^*_{\tilT}(\Flag_G)\cong Q
\end{equation*}
is given by
\begin{equation*}
\Pi(u,\xi,\eta,v)=(u,\xi,\eta,2v_1).
\end{equation*}
In particular,
\begin{equation*}
\pi_{\Flag}=\pi_Q\circ\Pi.
\end{equation*}
\end{prop}
\begin{proof}
The natural projection $\Pi$ is $\tilW$-equivariant and compatible with the left $\tilT$-equivariant structure, therefore the $(u,\xi)$-coordinates do not change. We proceed to show that the $\eta$-coordinate of theorem \ref{th:flagarr} coincides with the map $\pi_2$
\begin{equation*}
\pi_2: H^{\tilT}_2(\Flag_G)\to \frakt.
\end{equation*}
induced from $\pi_2$ in theorem \ref{th:cohoFlag}(2).
Since the $\tilW$-action comes from the left translations on $G(F)$ and $\pi_2$ comes from right translations on $G(F)$, the map $\pi_2$ is $\tilW$-invariant. In other words, $\pi_2$ factors through the coinvariants of $H^{\tilT}_2(\Flag_G)$ under the $\tilW$-action. Using theorem \ref{th:flagarr}(1), it is easy to see that the coinvariants of $H^{\tilT}_2(\Flag_G)$ under the $\tilW$-action are $\CC d\oplus\frakt_\eta$. Therefore, $\pi_2$ has the form
\begin{equation}\label{eq:factorpi2}
\pi_2=\pi_2'\circ(u,\eta): H^{\tilT}_2(\Flag_G)\stackrel{(u,\eta)}{\longrightarrow}\CC d\oplus\frakt\stackrel{\pi_2'}{\to}\frakt.
\end{equation}
On the other hand, the restriction of $\pi_{\Flag}=(\pi_0,\pi_1,\pi_2)$ to the subspace $V_{x_0}=\mu_{x_0}(\tilt)$
\begin{equation*}
\pi_{\Flag}\circ\mu_{x_0}:\tilt=H^{\tilT}_2(\{x_0\})=H^{\GG_m^{rot}}_2([T\backslash T/T])\to \CC d\oplus\frakt\to \CC\oplus\frakt\oplus\frakt
\end{equation*}
has the form
\begin{equation*}
\pi_{\Flag}\left(\mu_{x_0}(u,\xi)\right)=(u,\xi,\xi).
\end{equation*}
by direct computation. Plugging this into (\ref{eq:factorpi2}) we get
\begin{equation*}
\xi=\pi_2\left(\mu_{x_0}(u,\xi)\right)=\pi_2(u,\xi,\xi,0)=\pi_2'(u,\xi)
\end{equation*}
which shows that $\pi_2=\eta$. Finally the $v$-coordinate is determined by $(u,\xi,\eta)$ by theorem \ref{th:cone}. This completes the proof.
\end{proof}

\begin{remark}
In fact, the quadric $Q$ can be viewed as $N(Q_1,\frakt\times\frakt)$, where $Q_1$ is the quadric in $\frakt	\times\frakt$ defined by $|\xi|^2=|\eta|^2$. The above map becomes the natural map
$$N(\frakt\times_{\ch}\frakt,\frakt\times\frakt)\to N(Q_1,\frakt\times\frakt)$$
induced by the inclusion $\frakt\times_{\ch}\frakt\hookrightarrow Q_1$.
\end{remark}

\subsection{The convolution and comultiplication} We have a convolution diagram for $[I\backslash G(F)/I]$ as follows:
\begin{equation}\label{eq:flagconv}
\xymatrix{ & [I\backslash G(F)\stackrel{I}{\times}G(F)/I]\ar[dl]_{p_1}\ar[dr]^{p_2}\ar[rr]^{m} & & [I\backslash G(F)/I]\\
[I\backslash G(F)/I] & & [I\backslash G(F)/I]}
\end{equation}
where $m$ is induced by the multiplication of $G$. The rotation torus $\GG_m^{rot}$ acts on this diagram. This induces a comultiplication on $H^*_{\GG_m^{rot}}([I\backslash G(F)/I])$
\begin{equation}\label{eq:comult}
\Delta: H^*_{\GG_m^{rot}}([I\backslash G(F)/I])\to H^*_{\GG_m^{rot}}([I\backslash G(F)/I])\otimes_{H^*(\BB\tilT)} H^*_{\GG_m^{rot}}([I\backslash G(F)/I])
\end{equation}
where tensor product on the right hand uses the right $H^*(\BB\tilT)$-module structure of the first term and the left $H^*(\BB\tilT)$-module structure of the second term.

Taking spectra of the $\GG_m^{rot}$-equivariant cohomology of each term in the diagram (\ref{eq:flagconv}), we get
\begin{equation}\label{eq:specconv}
\xymatrix{ & N((\frakt/\ch)^3,\frakt^3)\ar[dl]_{p_{12}}\ar[dr]^{p_{23}}\ar[rr]^{p_{13}} & & N((\frakt/\ch)^2,\frakt^2)\\
N((\frakt/\ch)^2,\frakt^2) & & N((\frakt/\ch)^2,\frakt^2)}
\end{equation}
where $(X/Y)^n$ means the $n^\textup{th}$ self-fiber product of $X$ over $Y$.

With these structures, it is easy to show
\begin{prop}\label{p:groupoid}
\begin{enumerate}
 \item[]
 \item The affine scheme
$$\Spec H^*_{\GG_m^{rot}}([I\backslash G(F)/I])\cong N((\frakt/\ch)^2,\frakt^2)$$
admits a unique groupoid structure over $\AA^1\times\frakt$ with source map $(\pi_0,\pi_1)$, target map $(\pi_0,\pi_2)$ (in the notation of (\ref{eq:projflag})) and multiplication map induced by $\Delta$ in (\ref{eq:comult})
 \item the identity section of the groupoid is given by $\widetilde{\mu}_{x_0}$.
 \item the inversion map $\iota$ of the groupoid is given by
\begin{equation}\label{eq:defineiota}
\iota(u,\xi,\eta,v_i)=(u,\eta,\xi,-v_i).
\end{equation}
\end{enumerate}
\end{prop}

\subsection{Another involution} The stack $[I\backslash G(F)/I]$ admits an involution $\tau$ induced by $g\mapsto g^{-1}$. It is easy to see that $\tau$ induces an involution (also denoted by $\tau$) on $H^*_{\GG_m^{rot}}([I\backslash G(F)/I])$ by
\begin{equation}\label{eq:inv}
\tau(u,\xi,\eta,v_i)=(u,-\eta,-\xi,(-1)^{d_i}v_i).
\end{equation}
Note that $\tau$ differs from $\iota$ in (\ref{eq:defineiota}) by signs.

\subsection{Affine Weyl group actions}\label{ss:Weylaction}
Recall that by lemma \ref{lem:twistequiv}, $\tilW$ acts on $H^*_{\tilT}(\Flag_G)$. Let us denote this action by $\rho_1$. In terms of coordinates in \ref{p:flagexplicit}, we have for $\tilw=(\lambda,w)$
\begin{equation}\label{eq:leftaction}
\rho_1(\tilw)(u,\xi,\eta,v_i)=\left(u,w\xi+u\lambda,\eta,v_i+\frac{1}{u}(f_i(w\xi+u\lambda)-f_i(\xi))\right)
\end{equation}
Using the involution $\iota$, we get another action $\rho_2=\iota\rho_1\iota$ of $\tilW$ on $H^*_{\tilT}(\Flag_G)$ via
\begin{equation}\label{eq:rightaction}
\rho_2(\tilw)(u,\xi,\eta,v_i)=\left(u,\xi,w\eta+u\lambda,v_i-\frac{1}{u}(f_i(w\eta+u\lambda)-f_i(\eta))\right).
\end{equation}
These two actions commute with each other.

Define the group
\begin{equation}\label{eq:doubleW}
\tilW_{(2)}:=(\tilW\times\tilW)\rtimes\ZZ/2
\end{equation}
where $\ZZ/2$ is the group generated by the involution $\iota$ which acts on $\tilW\times\tilW$ by switching the two factors. Then we have an $\tilW_{(2)}$-action on $H^*_{\tilT}(\Flag_G)$ whose restriction on the first and second factors of $\tilW$ are $\rho_1$ and $\rho_2$ respectively.

Specializing to $\xi=0,u=0$, we get a $\tilW$-action on $H^*(\Flag_G)$ via $\rho_2$.
\begin{equation}\label{eq:actiononord}
\rho_2(\tilw)(\eta,v)=\left(w\eta,v_i-\partial_{w^{-1}\lambda}f_i(\eta)\right)=\left(w\eta,v_i-\partial_{\lambda}f_i(w\eta)\right).
\end{equation}
Note that the $\rho_1$-action on $H^*(\Flag_G)$ is trivial by lemma \ref{lem:twistequiv}.

\subsection{A more intrinsic description of $H_2^{\tilT}(\Flag_G)$}\label{ss:ultrah2} Consider the direct sum $\frakh\oplus\frakh$ of two copies of the affine Cartan algebra $\frakh$ with quadratic form $|x|^2-|y|^2$ and involution $(x,y)\mapsto(y,x)$ for $(x,y)\in \frakh\oplus\frakh$. The actions of $\tilW$ on the two copies of $\frakh$ and the involution assemble to a $\tilW_{(2)}$-action, which preserves the quadric form up to sign.

The vector $(K,K)\in\frakh_{\pm}$ is isotropic and fixed by the $\tilW_{(2)}$-action. We define
\begin{equation*}
\frakh_{(2)}:=(K,K)^{\bot}/\CC(K,K).
\end{equation*}
By construction, $\frakh_{(2)}$ inherits a quadratic form $(-|-)_2$ and a $\tilW_{(2)}$-action from those of $\frakh\oplus\frakh$. For each $\tilw\in\tilW$, the graph $\Gamma(\tilw)$ of $\tilw:\frakh\to\frakh$ is a subspace of $(K,K)^{\bot}$.

\begin{prop}\label{p:ultrah2}
\begin{enumerate}
 \item []
 \item Up to sign, there is a unique $\tilW_{(2)}$-equivariant isomorphism of quadratic spaces
\begin{equation*}
(\frakh_{(2)},(-|-)_2)\stackrel{\sim}{\to}(H_2^{\tilT}(\Flag_G),q)
\end{equation*}
where the $\tilW_{(2)}$-structure on $H_2^{\tilT}(\Flag_G)$ is defined in section \ref{ss:Weylaction}.
 \item Under any such isomorphism, the fixed point subspace $V_{\tilw\cdot x_0}$ corresponding to the fixed point $\tilw\cdot x_0$ is the image of the graph $\Gamma(\tilw^{-1})$ in $\frakh_{(2)}$.
\end{enumerate}
\end{prop}

\begin{proof}
(1) We have a natural decomposition
\begin{equation}\label{eq:decomph2}
\frakh_{(2)}=\CC(d,d)\oplus\frakt\oplus\frakt\oplus\CC(0,K)
\end{equation}
The sought-for isomorphism is the composition of $\psi_*$ with the isomorphism
\begin{eqnarray*}
\frakh_{(2)}=\CC(d,d)\oplus\frakt\oplus\frakt\oplus\CC(0,K) & \to & \CC d\oplus\frakt\oplus\frakt\oplus\CC K\\
u(d,d)+(\xi,0)+(0,\eta)+v(0,K) & \mapsto & (u,\xi,\eta,v)
\end{eqnarray*}
It is easy to check this isomorphism sends $(-|-)_{(2)}$ to $q$ and is $\tilW_{(2)}$-equivariant. This shows the existence part.

For the uniqueness, we only have to show that $\frakh_{(2)}$ does not have $\tilW_{(2)}$-equivariant orthogonal automorphisms other than $\pm1$. Let $\alpha$ be such an automorphism. Decomposing $\frakh_{(2)}$ into eigenspaces of the involution, we get
\begin{eqnarray*}
\frakh^+_{(2)}=\CC(d,d)\oplus\Delta^+(\frakt)\\
\frakh^-_{(2)}=\CC(0,K)\oplus\Delta^-(\frakt)
\end{eqnarray*}
where $\Delta^+(\frakt)$ and $\Delta^-(\frakt)$ are the diagonal and anti-diagonal in $\frakt\oplus\frakt$. These eigenspaces $\frakh^+_{(2)}$ and $\frakh^-_{(2)}$ must be preserved by $\alpha$. Moreover, if we look at the $W\times W$-action, we decide that each copy of $\frakt$ in $\frakh_{(2)}$ should be preserved by $\alpha$, and $\alpha$ acts on each copy by a scalar(because we assumed that $G$ is simple). These two facts together show that $\alpha$ perserves the decomposition of (\ref{eq:decomph2}) and acts on each direct summand by a scalar. It is not difficult to see that these scalars have to be the same in order for $\alpha$ to be $\tilW_{(2)}$-equivariant. Since $\alpha$ preserves the quadratic form $(-|-)_2$, the scalar has to be $\pm1$.

(2) follows from (1) and theorem \ref{th:flagarr}.
\end{proof}

\subsection{Linear subspaces in a quadric cone}\label{ss:linearinquadric}
By theorem \ref{th:cone}, the fixed point arrangement for $\Flag_G$ consists of maximal isotropic subspace with respect to the quadratic form $q$ in (\ref{eq:quadric}). Recall some general facts about such linear subspaces in a quadric cone. Let $(E,q)$ be a quadratic space over $\CC$. Assume $q$ is nondegenerate and $\dim E=2n$. Let $Q$ be the quadric cone in $V$ defined by $q=0$. Let $\Sigma$ be the scheme parametrizing all $n$-dimensional linear subspaces of $E$ which lie on $Q$. This is a closed subscheme of the Grassmannian $\Grass(E,n)$. It is a homogeneous space under the orthogonal group $O(E,q)$. It has two connected components $\Sigma^{\pm}$ (there is no canonical labelling of the two components) which are permuted by the two components of $O(E,q)$.

Back to our situation. Define $\Rig(\frakt)=\frakt\rtimes O(\frakt,(-|-))$ to be the group of rigid affine transformations on $\frakt$ with respect to the Killing form $(-|-)$. It acts on $H^{\tilT}_2(\Flag_G)$ via a similar formula as (\ref{eq:4termaction})
\begin{equation}\label{eq:rigaction}
(\lambda,S)\cdot(u,\xi,\eta,v)=\left(u,S\xi+u\lambda,\eta,v+(S\eta|\lambda)+\frac{u}{2}|\lambda|^2\right) 
\end{equation}
where $\lambda\in\frakt$ and $S\in O(\frakt,(-|-))$. Note that $\tilW\subset\Rig(\frakt)$, and the action $\rho_1$ of $\tilW$ coincides with the restriction from the $\Rig(\frakt)$-action.

\begin{prop}\label{p:linearinquadric}
\begin{enumerate}
 \item[]
 \item There is a unique dense open subset $U$ of $\Sigma$ which is a $\Rig(\frakt)$-torsor;
 \item If we call $\Sigma^+$ the component of $\Sigma$ containing $V_{x_0}$, then $V_{\tilw\cdot x_0}$ is contained in $\Sigma^+$ (resp. $\Sigma^-$) if $\ell(w)$ is even (resp. odd). 
\end{enumerate}
\end{prop}
\begin{proof}
(1) Consider the affine open subset $U'$ of $\Grass(H^{\tilT}_2(\Flag_G),r+1)$ consisting of $(r+1)$-dimensional subspaces transversal to $\frakt_\xi\oplus\CC K$. This is a Schubert cell in $\Grass(H^{\tilT}_2(\Flag_G),r+1)$ and we have a natural isomorphism
\begin{equation*}
U'\cong \Hom(\CC d\oplus\frakt_{\eta},\frakt_{\xi}\oplus\CC K)
\end{equation*}
The group $\Rig(\frakt)$ embeds into this $\Hom$-space by  
\begin{equation*}
(\lambda,S)\mapsto\left((u,\eta)\mapsto\left(S\eta+u\lambda,(S\xi|\lambda)+\frac{u}{2}|\lambda|^2\right)\right).
\end{equation*}
It is easy to check that the image of the embedding is the $\Rig(\frakt)$-orbit of $V_{x_0}\in U'$, and $U=U'\cap\Sigma$ coincide with this orbit. Therefore $U$ is a $\Rig(\frakt)$-torsor.

(2) The group $\Rig(\frakt)$ has two components, given by the determinant of the linear part. For $\tilw=(\lambda,w)\in\tilW$, this determinant is $(-1)^{\ell(w)}$. By (1), $\Rig(\frakt)$ permutes the two components $U\cap\Sigma^{\pm}$ via the determinant($=\pm 1$) of its linear part, and (2) is immediate.
\end{proof}

% section on the proof of the main theorem

\section{Proof of theorem \ref{th:flagarr}}\label{s:proof}

This section contains the proof of theorem \ref{th:flagarr}. We need a technical lemma.

\begin{lemma}\label{lem:p1action}
For $i=0,\cdots,r$, there is a (not necessarily unique) $\tilT$-equivariant isomorphism
\begin{equation*}
\gamma_i:C_i\stackrel{\sim}{\longrightarrow}\PP^1
\end{equation*}
sending $x_0$ to $[0,1]$ and $s_ix_0$ to $[1,0]$. Here $\tilT$ acts on $\PP^1$ via the character $-\alpha_i$, i.e.,
\begin{equation*}
t\cdot[x,y]=[\alpha_i(t)^{-1}x,y],\forall t\in\tilT.
\end{equation*}
\end{lemma}
\begin{proof}
For any root $\alpha$, let $\frakg_\alpha\subset \frakg$ be the corresponding root space and $X_\alpha\in\frakg_\alpha$ be a nonzero root vector. For $i=1,\cdots,r$ we have
\begin{equation*}
Is_i\cdot x_0=\exp(\frakg_{-\alpha_i}).
\end{equation*}
From this it is easy to see that $\tilT$ acts on $C_i$ via $-\alpha_i$. Let $\gamma_i$ be the unique extension to $C_i$ of the map
\begin{eqnarray*}
Is_i\cdot x_0=\exp(\frakg_{-\alpha_i}) &\to& \AA^1\subset\PP^1\\
\exp(sX_{-\alpha_i}) &\mapsto& [s,1].
\end{eqnarray*}
For $i=0$, we have
\begin{equation*}
Is_0\cdot x_0=\exp(z^{-1}\frakg_{\theta}).
\end{equation*}
It is easy to see that $\tilT$ acts on $C_0$ via $-\alpha_0=-\delta+\theta$. Let $\gamma_0$ be the unique extension to $C_0$ of the map
\begin{eqnarray*}
Is_0\cdot x_0=\exp(z^{-1}\frakg_{\theta}) &\to& \AA^1\subset\PP^1\\
\exp(sz^{-1}X_{\theta}) &\mapsto& [s,1].
\end{eqnarray*}
It is easy to verify that these $\gamma_i$'s satisfy the requirements.
\end{proof}

\subsection{Proof of Theorem \ref{th:flagarr}}
\begin{proof}
It is clear that the conditions (\ref{eq:mux0}) and (\ref{eq:diagforpsi}) determines $\psi_*$ uniquely. Moreover, from (\ref{eq:4termaction}) and lemma \ref{lem:twistequiv}, it is easy to see that all the squares in (\ref{eq:diagforpsi}) commute. To prove the theorem, we need to show that the $\psi_*$ thus determined is $\tilW$-equivariant. We will proceed the proof along with the calculation of the sections $\mu_{\tilw\cdot x_0}$.

For each $\tilw\in\tilW$, by lemma \ref{lem:twistequiv}, the action of $\tilw$ on $H^{\tilT}_2(\Flag_G)$ has the form:
\begin{equation}\label{eq:defineB}
\tilw\cdot(u,\xi,\eta,v)=\psi_*\left(u,w\xi+u\lambda,\eta,v)+\psi_*(0,0,B_{\tilw}(u,\xi)\right)
\end{equation}
for some linear map $B_{\tilw}:\CC d\oplus\frakt\to\frakt\oplus\CC K$. The fact that $B_{\tilw}$ comes from an action implies that the assignment $\tilw\mapsto B_{\tilw}$ satisfies the cocycle condition
\begin{equation*}
B_{\tilw_1\tilw_2}=B_{\tilw_1}\circ\tilw_2+B_{\tilw_2}, \forall \tilw_1, \tilw_2\in\tilW .
\end{equation*}
In fact, this is the cocycle defining the extension (\ref{eq:homotwoexact}) of $\tilW$-modules.

On the other hand, consider linear maps
\begin{equation}\label{eq:defineA}
A_{\tilw}:=\mu_{\tilw\cdot x_0}-\mu_{x_0}: \tilt\to H_2(\Flag_G).
\end{equation}
By lemma \ref{c:twistmu}, it is easy to check that
\begin{equation}\label{eq:relateAB}
\psi_*^{-1}(A_{\tilw})=B_{\tilw}\circ\tilw^{-1}+w^{-1}(\xi-u\lambda)-\xi.
\end{equation}
A direct calculation shows that the assignment $\tilw\mapsto A_{\tilw}$ also satisfies a similar cocycle condition
\begin{equation*}
 A_{\tilw_1\tilw_2}=A_{\tilw_2}\circ\tilw_1^{-1}+A_{\tilw_1}, \forall \tilw_1, \tilw_2\in\tilW .
\end{equation*}

We can calculate $A_{s_i}$ explicitly. According to lemma \ref{lem:p1action}, we have a commutative diagram:
\[
\xymatrix{\tilt\ar[rr]^{A_{s_i}}\ar[d]^{-\alpha_i} & & H_2(C_i)\ar[d]^{\gamma_{i,*}}\ar[r] & H_2(\Flag_G)\\
\CC\ar[rr]^{\mu_{\infty}-\mu_0} & & H_2(\PP^1)}
\]
Here, $\mu_0$ and $\mu_\infty$ are the sections as in (\ref{eq:definesection}) corresponding to the fixed points $0=[0,1]$ and $\infty=[1,0]$ of the following $\GG_m$-action on $\PP^1$:
\begin{equation*}
s\cdot[x,y]=[sx,y], \forall s\in \GG_m, [x,y]\in\PP^1.
\end{equation*}
Standard calculation shows that $(\mu_\infty-\mu_0)(1)=[\PP^1]$. Therefore, using the coordinates in the decomposition (\ref{eq:decomp}), we conclude that
\begin{equation*}
A_{s_i}(u,\xi)=\langle(u,\xi),-\alpha_i\rangle[C_i], \forall i=0,\cdots,r.
\end{equation*}
or equivalently
\begin{equation*}
\phi_{*,\CC}\circ A_{s_i}(u,\xi)=\langle(u,\xi),-\alpha_i\rangle\check{\alpha_i}=s_i(u,\xi)-(u,\xi), \forall i=0,\cdots,r.
\end{equation*}
Because $\{A_{\tilw}\}$ is a cocycle, it is unique once $\{A_{s_0},\cdots,A_{s_r}\}$ are determined. Therefore it is not difficult to check that
\begin{equation}\label{eq:solveA}
\phi_{*,\CC}\circ A_{\tilw}(u,\xi)=\tilw^{-1}(u,\xi)-(u,\xi), \forall \tilw\in\tilW.
\end{equation}
Hence by (\ref{eq:defineA})
\begin{equation*}
\pi_1\circ\mu_{\tilw\cdot x_0}(u,\xi)=\xi+A_{\tilw}(u,\xi)=\tilw^{-1}(ud+\xi)-ud
\end{equation*}
By (\ref{eq:weylaction}), this is exactly $\left(w^{-1}(\xi-u\lambda),(\xi|\lambda)-\frac{u}{2}|\lambda|^2\right)$. This proves the second statement of the theorem.

Plugging (\ref{eq:solveA}) into (\ref{eq:relateAB}), we get
\begin{equation*}
B_{\tilw}(u,\xi)=\left((w\xi|\lambda)+\frac{u}{2}|\lambda|^2\right)K.
\end{equation*}
Comparing with (\ref{eq:4termaction}), we conclude that $\psi_*$ is $\tilW$-equivariant. This completes the proof of the theorem.
\end{proof}

%% moment maps

\section{Application to moment maps}\label{s:moment}

In this section, we will use our result on the fixed point arrangement of $\Flag_G$ to study moment maps of $\Flag_G$ under the $\tilT$-action. First we need to recall a general procedure to recover the moment map image of fixed points from the fixed point arrangement. For this purpose, we will work with a torus $T$ acting on a projective scheme $X$ with finitely many fixed points. As we are talking about moment maps, we restrict the action to the {\em compact} form $T_\RR$ of $T$. Let $\frakt_\RR$ be the Lie algebra of $T_\RR$.

\subsection{The universal moment map}

For each fixed point $x\in X^T$, we have the restriction maps
\begin{equation}\label{eq:mudual}
i_{x}^*: H^2_T(X,\RR)\to H^2_T(\{x\},\RR)\cong\frakt^*_\RR 
\end{equation}
which is dual to the real form the section $\mu_x$ in (\ref{eq:definesection}). The collection of $i_x^*$ for all $x\in X^T$ gives a map:
\begin{equation}\label{eq:univmom}
 \mathfrak{M}_X: X^T\times H^2_T(X,\RR)\to\frakt^*_\RR
\end{equation}
which we call the {\em universal moment map}. The name is justified as we will see in \ref{p:momentagree} that the effect of all moment maps on the fixed points can be read from the map $\mathfrak{M}_X$. 

\subsection{Standard moment maps}\label{ss:stmoment}
Let $E$ be a $\CC$-vector space and $\PP(E^*)$ be the projective space parametrizing lines in $E$. Suppose a torus $T$ acts on $E$ linearly, we get a canonical $T$-action on $\PP(E^*)$ and an $T$-equivariant structure on $\calO_{\PP(E^*)}(1)$. If we choose the (properly normalized) Fubini-Study form $\omega_{FS}$ on $\PP(E^*)$, we can arrange the moment map
\begin{equation}\label{eq:stmoment}
m_{st}: \PP(E^*)\to \frakt^*_\RR
\end{equation}
in such a way that for any character $\chi\in X^\bullet(T)$ and any $\chi$-eigenline $\ell\subset E$, we have $m_{st}([\ell])=\chi$. The map $m_{st}$ is called the {\em standard moment map} associated to the $T$-equivariant structure on $\calO_{\PP(E^*)}(1)$.   

Let $\Pic^T(X)$ denote the isomorphism classes of $T$-equivariant line bundles on $X$. If $\calL\in \Pic^T(X)$ defines a $T$-equivariant projective morphism $f:X\to \PP(|\calL|)=\PP^n$, then $\calO_{\PP^n}(1)$ carries a unique $T$-equivariant structure. The standard moment map on $\PP^n$, pulled back to $X$, gives a moment map
\begin{equation}\label{eq:definemoment}
m_{\calL}:X\stackrel{f}{\rightarrow} \PP^n \stackrel{m_{st}}{\longrightarrow} \frakt^*_\RR
\end{equation}
This map is the same as the moment map defined by the K\"{a}hler form $f^*\omega_{FS}$.

\begin{prop}\label{p:momentagree}
For $\calL\in \Pic^T(X)$ as above, we have
\begin{equation}\label{eq:momentagree}
 \mathfrak{M}_X\left(x,c^T_1(\calL)\right)=m_{\calL}(x)
\end{equation}
for all $x\in X^T$. Here $c^T_1:\Pic^T(X)\to H^2_T(X)$ is the equivariant Chern class.
\end{prop}
\begin{proof}
We have a commutative diagram
\begin{equation*}
\xymatrix{X^T\ar@<-6ex>[d]^{f}\times H^2_T(X,\RR)\ar[r]^(.7){\mathfrak{M}_{X}} & \frakt^*_\RR \ar@{=}[d]\\
(\PP^n)^T\times H^2_T(\PP^n,\RR)\ar[r]^(.7){\mathfrak{M}_{\PP^n}}\ar@<-2ex>[u]^{f^*} & \frakt^*_\RR}
\end{equation*}
Hence for each $x\in X^T$, we have
\begin{equation*}
\mathfrak{M}_X\left(x,c_1^T(\calL)\right)=\mathfrak{M}_{\PP^n}\left(f(x),c_1^T(\calO_{\PP^n}(1))\right). 
\end{equation*}
In view of the definition (\ref{eq:definemoment}) of $m_\calL$, it suffices to prove (\ref{eq:momentagree}) for $X=\PP^n=\PP(E^*)$ and $\calL=\calO_{\PP(E^*)}(1)$. Now each fixed point $x$ corresponds to an eigenline $\ell\subset E$ on which $T$ acts via some character $\chi\in X^\bullet(T)$. Consider the Cartesian diagram
\begin{equation*}
\xymatrix{\EE T\stackrel{T}{\times}(\ell-\{0\})\ar[r]\ar[d] & \EE T\stackrel{T}{\times}(E-\{0\})\ar[d]^{p}\\
\EE T\stackrel{T}{\times}(\{x\})\ar[r]^{i_x} & \EE T\stackrel{T}{\times}\PP(E^*)}
\end{equation*}
where $p$ realizes the total space of the line bundle $\calO_{\PP(E^*)}(-1)$. From this we see that $i_x^*\calO_{\PP(E^*)}(1)$ is the line bundle $\calO(\chi)$ on $\BB T$ with total space $\EE T\stackrel{T}{\times}\CC(-\chi)$. Therefore we have
\begin{equation*}
\mathfrak{M}_{\PP(E^*)}\left(x,c^T_1(\calO_{\PP(E^*)}(1))\right)=c^T_1(\calO(\chi))=\chi=m_{st}(x).
\end{equation*}
This completes the proof of the proposition.
\end{proof}

\begin{remark}
The above discussion gives the procedure to pass from the fixed point arrangement to the moment map images of fixed points. Although we restricted ourselves to projective schemes, the procedure also works for ind-projective schemes once the line bundle defines an ind-projective morphism. 
\end{remark}

\subsection{Line bundles on $\Flag_G$}\label{ss:pic} For each weight $\chi\in X^\bullet(T)$, we have a line bundle $\calL_\chi=G(F)\stackrel{I}{\times}\CC(-\chi)$. Here $\CC(-\chi)$ is the one-dimensional representation of $I$ through the quotient $I\twoheadrightarrow B\twoheadrightarrow T$. Moreover, we have the {\em determinant bundle} $\calL_{\det}$ which is the pull-back of the ample generator $\calL_{\det,\Grass}$ of $\Pic(\Grass_G)$ (see \cite{LS}).

We have a natural {\em degree pairing}
\begin{equation}\label{eq:definepairing}
\langle-,-\rangle: \Pic(\Flag_G)\times H_2(\Flag_G,\ZZ)\to\ZZ 
\end{equation}
which sends $(\calL,[C])$ to $\deg(\calL|_{[C]})=\langle c_1(\calL),[C]\rangle$.

\begin{lemma}\label{lem:pic}
We have an isomorphism
\begin{equation}\label{eq:phipic}
\phi^*: X^\bullet(T)\oplus\ZZ\Lambda_0\stackrel{\sim}{\longrightarrow}\Pic(\Flag_G)
\end{equation}
which sends $\chi\in X^\bullet(T)$ to $\calL_{\chi}$ and $\Lambda_0$ to $\calL_{\det}$. Moreover, the isomorphism $\phi_*$ in (\ref{eq:homotwo}) and $\phi^*$ are compatible with the natural pairing between $\frakh$ and $\frakh^*$ and the degree pairing (\ref{eq:definepairing}), i.e.,
\begin{equation}\label{eq:pairingagree}
\langle\calL,[C]\rangle=\langle\phi^{*,-1}(\calL),\phi_*^{-1}([C])\rangle.
\end{equation}
\end{lemma}

\begin{proof}
It suffices to check (\ref{eq:pairingagree}), because then $\phi^*$ is uniquely determined as the dual of $\phi_*$. For $i=1,\cdots,r$ there is a homomorphism
\begin{equation*}
\iota_i: \SL(2)\to G\subset G(F)
\end{equation*}
such that the induced map on Lie algebras sends the standard basis $\{e,f,h\}$ of $\mathfrak{sl}(2)$ to $\{X_{\alpha_i}, X_{-\alpha_i}, \check{\alpha_i}\}$. Here for a root $\alpha$, $X_{\alpha}$ is some nonzero vector in the root space $\frakg_{\alpha}$ (c.f. the proof of \ref{lem:p1action}). This induces an isomorphism
\begin{equation*}
\iota^*_i: \calL_{\chi}\cong \SL(2)\stackrel{B}{\times}\CC(\langle-\chi,\check{\alpha_i}\rangle)
\end{equation*}
which has degree $\langle\chi,\check{\alpha_i}\rangle$ viewed as a line bundle on $\PP^1=\SL(2)/B\stackrel{\iota_i}{\to}C_i$ ($B$ is the Borel of $\SL(2)$ containing $\exp(e)$).

For $i=0$, we have a homomorphism
\begin{equation*}
\iota_0: \SL(2)\to G(F)
\end{equation*}
such that the induced map on Lie algebras sends $\{e,f,h\}$ to $\{zX_{-\theta}, z^{-1}X_{\theta}, -\check{\theta}\}$. This induces an isomorphism
\begin{equation*}
\iota^*_0: \calL_{\chi}\cong \SL(2)\stackrel{B}{\times}\CC(\langle-\chi,-\check{\theta}\rangle)
\end{equation*}
which has degree $-\langle\chi,\check{\theta}\rangle=\langle\chi,\check{\alpha_0}\rangle$ viewed as a line bundle on $\PP^1=\SL(2)/B\stackrel{\iota_0}{\to}C_0$. 

Finally, for $\calL_{\det}$, we have by adjunction
\begin{equation*}
\langle \calL_{\det}, [C_i]\rangle=\langle\calL_{\det,\Grass},p_{G,*}[C_i]\rangle=\delta_{i,0}=\langle\Lambda_0,\check{\alpha_i}\rangle
\end{equation*}
because $p_{G,*}[C_i]=0$ for $i=1,\cdots,r$ and $p_{G,*}[C_0]$ generates $H_2(\Grass, \ZZ)$. This completes the proof of (\ref{eq:pairingagree}), hence the lemma.  
\end{proof}

We are now ready to write down the images of $\Flag_G^T$ under the moment map $m_{\calL}$ associated to a line bundle $\calL=\calL_{\chi}\otimes\calL_{\det}^{\otimes\kappa}$. We have an isomorphism
\begin{equation}\label{eq:psidual}
\psi^*: \CC\delta\oplus\frakt^*\oplus\frakt^*\oplus\CC\Lambda_0\cong H_{\tilT}^2(\Flag_G)
\end{equation}
dual to $\psi_*$ in (\ref{eq:definepsi}).
We can fix the $\tilT$-equivariant structure of $\calL$ by requiring
\begin{equation}\label{eq:imagex0}
m_{\calL}(x_0)=(0,\chi)\in\RR\delta\oplus\frakt^*_\RR\cong\tilt^*_\RR.
\end{equation}
Different choices of equivariant structures on $\calL$ only result in a translation of the moment map image.

\begin{prop}\label{p:momentcoor}
In terms of the decomposition (\ref{eq:tiltdual}), for $\tilw=(\lambda,w)\in\tilW$, the image of $\tilw\cdot x_0$ under the moment map $m_{\calL}$ is
\begin{equation}\label{eq:momentcoor}
m_{\calL}(\tilw\cdot x_0)=\left(-\langle w\chi,\lambda\rangle-\frac{\kappa}{2}|\lambda|^2, w\chi+\kappa\sigma(\lambda)\right)
\end{equation}
where $\sigma:\frakt\stackrel{\sim}{\to}\frakt^*$ is as in (\ref{eq:definesigma}).
\end{prop}
\begin{proof}
Since $\calL=\calL_\chi\otimes \calL_{\det}^{\otimes\kappa}$, by \ref{lem:pic}, we have
\begin{equation*}
c^T_1(\calL)=(-,-,\chi,\kappa)
\end{equation*}
in terms of the coordinates given by the decomposition (\ref{eq:psidual}). The requirement (\ref{eq:imagex0}) shows that the first two coordinate are actually 0. Taking the dual of $\mu_{\tilw\cdot x_0}$ (as written explicitly in (\ref{eq:formulaformu})), we get
\begin{eqnarray*}
& & \mathfrak{M}_{\Flag_G}(\tilw\cdot x_0,c^T_1(\calL))=\mathfrak{M}_{\Flag_G}\left(\tilw\cdot x_0,(0,0,\chi,\kappa)\right)\\
&=& \left(-\langle w\chi,\lambda\rangle-\frac{\kappa}{2}|\lambda|^2, w\chi+\kappa\sigma(\lambda)\right)
\end{eqnarray*}
which gives (\ref{eq:momentcoor}) by proposition \ref{p:momentagree}. 
\end{proof}

\begin{cor}\label{c:paraboloid}
The image of $\Flag_G^T$ under the moment map $m_{\calL}$ lies on the paraboloid 
\begin{equation}\label{eq:paraboloid}
\{(m_0,m_1)\in\RR\delta\oplus\frakt_\RR^*|m_0=-\dfrac{1}{2\kappa}(|m_1|^2-|\chi|^2)\}.
\end{equation}
\end{cor}
\begin{proof}
This can be verified using proposition \ref{p:momentcoor} by direct calculation.
\end{proof}

\end{document}